\documentclass[12pt,a4paper]{article}
\usepackage[utf8]{inputenc}
\usepackage{amsmath}
\usepackage{amsfonts}
\usepackage{amssymb}
\usepackage{amsthm}
\usepackage{float}
\usepackage{hyperref}
\usepackage[left=2cm,right=2cm,top=2cm,bottom=2cm]{geometry}
\usepackage{xcolor}
\usepackage[all,cmtip]{xy}
\usepackage{tikz}
\usepackage{tikz-cd}
\usetikzlibrary{decorations.markings}
\usepackage{enumerate}
\usepackage[normalem]{ulem}

\newcommand{\Hom}{\mathrm{Hom}}

\newcommand{\Th}{\mathrm{Th}}

\newcommand{\rk}{\mathrm{rk}}

\newcommand{\im}{\mathrm{im}\,}

\newcommand{\RR}{\mathbb{R}}
\newcommand{\QQ}{\mathbb{Q}}
\newcommand{\CC}{\mathbb{C}}
\newcommand{\ZZ}{\mathbb{Z}}

\newtheorem{thm}{Theorem}[section]
\newtheorem{prop}[thm]{Proposition}
\newtheorem{cor}[thm]{Corollary}
\newtheorem{lem}[thm]{Lemma}

\theoremstyle{definition}
\newtheorem{rem}[thm]{Remark}
\newtheorem{defn}[thm]{Definition}
\newtheorem{ex}[thm]{Example}

\begin{document}

\title{On orientability, Poincaré duality, and connectivity of GKM graphs}
\author{Oliver Goertsches\footnote{Philipps-Universit\"at Marburg, email:
goertsch@mathematik.uni-marburg.de}, Panagiotis Konstantis\footnote{Philipps-Universität Marburg,
email: pako@mathematik.uni-marburg.de}, and Leopold
Zoller\footnote{Universitat de Barcelona, email: leopold.zoller@ub.edu}}

\maketitle

\begin{abstract} 
We investigate a combinatorial notion of orientability for abstract GKM graphs and its connections to graph cohomology in the sense of Guillemin--Zara. In particular, we prove that orientability of the GKM graph is equivalent to Poincaré duality of the rational (non-equivariant) graph cohomology algebra. As an application, we prove that orientable GKM graphs remain connected after removing any single vertex.
\end{abstract}

\section{Introduction}
A GKM manifold is a compact simply-connected manifold $M$ together with the action of a compact torus $T$, such that the orbits of dimension $\leq 1$ form a graph of $2$-spheres, glued together at the fixed points. This graph can, as a $T$-space, be encoded in a purely combinatorial labeled graph, the so-called GKM graph $(\Gamma,\alpha)$. Furthermore one asks for odd cohomology of a GKM manifold to vanish (say with rational coefficients), as in this case Goresky-Kottwitz-MacPherson observed in \cite{MR1489894} that the GKM graph encodes the cohomology ring of the manifold through a relatively simple algebraic formula. This rather flexible geometric setup applies to many classical examples such as toric manifolds or certain compact homogeneous spaces.

While geometry provides the key motivation for GKM theory, in this paper we are working purely on the combinatorial side. In \cite{GuilleminZaraEquivariantdeRham, GuilleminZara1Skeleta} the notion of an \emph{abstract GKM graph} was introduced as certain labeled graph (see Definition \ref{defn: abstract GKM graph}) and the combinatorial objects were studied independently from geometry. In particular the aforementioned formula for the cohomology leads to the definition of the equivariant (resp.\ non-equivariant) graph cohomology $H_T^*(\Gamma)$ (resp.\ $H^*(\Gamma)$) for an abstract GKM graph $(\Gamma,\alpha)$. In case $(\Gamma,\alpha)$ is the GKM graph of a GKM manifold $M$, the graph cohomology $H_T^*(\Gamma)$ (resp.\ $H^*(\Gamma)$) agrees with $H_T^*(M)$ (resp.\ $H^*(M)$).

One possible viewpoint that makes the independent study of abstract GKM graphs interesting is the question of realizability: while any GKM manifold has a GKM graph associated to it, it is not known which abstract GKM graphs come from GKM manifolds. Or in other words, one may ask for a list of combinatorial conditions on an abstract GKM graph that are equivalent to being realizable by a manifold. Studying the combinatorial object and comparing it with the behaviour in the geometric setup of a GKM manifold in particular serves to answer the question whether the combinatorial object is ``the right one''.

In \cite{2210.01856v1} it was shown that a 3-valent abstract GKM graph $(\Gamma,\alpha)$ with labels in $\ZZ^2$ is realizable (by a $6$-dimensional simply-connected rational $T^2$-GKM manifold) if and only if $H^*(\Gamma)$ satisfies Poincaré duality. The Poincaré duality condition is evidently necessary as for any realization $M$ we will have $H^*(\Gamma)\cong H^*(M)$. However, there is another more fundamental property, called \emph{orientability} of the GKM graph (see Definition \ref{D: path orientability}) which shows up in the initial stages of the construction of the geometric realization. It is designed to allow for the construction of a thickening of the graph of $2$-spheres (corresponding to the GKM graph) to a $T$-manifold with boundary which is orientable in the geometric sense.
In \cite[Corollary 2.8]{2210.01856v1} it was shown, in the setting of $3$-valent $T^2$-GKM graphs, that the Poincaré duality condition implies orientability. As Poincaré duality was needed in later stages of the realization process, it thus remained as the only necessary assumption for realizability.  

The first goal of the present article is to study {this} notion of orientability and explore its relations to graph cohomology. Our first main result is

\begin{thm}\label{thm:orientabilities}
	Let $\Gamma$ be a $d$-valent abstract $T^k$-GKM graph such that $H^*_T(\Gamma)$ is a free $H^*(BT)$-module. Then the following conditions are equivalent:
	\begin{enumerate}
		\item $\Gamma$ is orientable.
		\item $H^{2d}(\Gamma)\neq 0$.
		\item $\dim H^{2d}(\Gamma)=1$ and $H^{2d}(\Gamma)$ is generated by the Thom class of any vertex of $\Gamma$.
		\item $H^*(\Gamma)$ satisfies $2d$-dimensional Poincaré duality.
	\end{enumerate}
\end{thm}

	As an immediate consequence we can rewrite the conditions from \cite[Theorem 1.1]{2210.01856v1} for the geometric realization of $3$-valent
$T^{2}$-graphs in noncohomological terms. We refer to \cite{2210.01856v1} for the detailed geometric definitions in the statement, which are not covered here.

\begin{cor}\label{C: realization without pd}
A $3$-valent abstract $T^{2}$-GKM graph is realizable by a $6$-dimensional simply-connected rational $T^{2}$-GKM
manifold if and only if it is orientable.
\end{cor}
Indeed, by \cite[Proposition
2.25]{2210.01856v1} (cf. Remark 4.2) the equivariant graph cohomology of a $T^k$-GKM graph is automatically free for $k=2$, so Theorem \ref{thm:orientabilities} lets us replace the aforementioned Poincaré duality condition by orientability. For $k\geq 3$, the freeness condition in Theorem \ref{thm:orientabilities} is necessary for the theorem to hold (see Example \ref{ex: orientable non-free}) and also a necessary condition for realizability by any GKM manifold $M$ as $H_T^*(M)$ is a free $H^*(BT)$-module. The fact that top nonequivariant graph cohomology $H^{2d}(\Gamma)$ is either zero- or one-dimensional was previously proved under more restrictive assumptions in \cite[Proposition 4.22]{Luo} (see Remarks \ref{rem:Luo1} and \ref{rem:Luoassumptions}). Furthermore, the fact that the Thom class of a vertex is nontrivial in nonequivariant graph cohomology follows from results of \cite{GuilleminZaraEquivariantdeRham} in the more restrictive case of a signed GKM graph (see Remarks \ref{rem: signed} and \ref{rem: signed orientable}).

The proof of Theorem \ref{thm:orientabilities} involves the introduction of some concepts that we consider to be interesting in their own right. To prove the implication ``$1.\Rightarrow 4.$'' we introduce for any subgraph $G\subset \Gamma$ equivariant graph cohomology $H^*_T(G)$ and equivariant graph cohomology with compact support $H^*_{T,c}(G)$ (see Section \ref{sec:PD}). The ABBV formula introduces a well-defined integration on graph cohomology with values in $H^*(BT)=:R$ (Theorem \ref{thm: ABBV}) and we obtain a form of Poincaré duality on equivariant graph cohomology (see Theorem \ref{thm:PD}):
\begin{thm}
Let $\Gamma$ be an orientable GKM graph and $G\subset \Gamma$ any subgraph. Then \[
H^{*}_T(G)\longrightarrow \Hom_R(H^{*}_{T,c}(G),R);\qquad \alpha\mapsto (\beta\mapsto \int (\alpha\cdot \beta)),
\]
is an isomorphism.
\end{thm}

The proof is an inductive procedure using Mayer-Vietoris sequences for graph cohomology and is much inspired by the proof of classical Poincaré duality for orientable manifolds. The above result on equivariant graph cohomology then specializes to the implication ``$1.\Rightarrow 4.$'' on non-equivariant cohomology as shown in Corollary \ref{cor:nonequivPD}.
For the implication ``$2.\Rightarrow 1.$'' we introduce the concept of an orientability covering of a GKM graph and study its properties.

As an application of the techniques developed surrounding the proof of Poincaré duality we furthermore obtain connectivity results on orientable GKM graphs (see Theorem \ref{thm:connectivity}).

\begin{thm}\label{thm: intro connectivity}
Any orientable $d$-valent GKM graph $(\Gamma,\alpha)$ remains connected after removing any one vertex with its adjacent edges. If $d\geq 2$, it also remains connected after removing any one edge.
\end{thm}

Related connectivity results were obtained by Luo \cite{Luo} under substantially stronger
assumptions. The hypotheses in \cite{Luo} imply stronger connectivity properties,which are in general not fulfilled by orientable GKM graphs (see Example \ref{ex:not3connected} and Remark \ref{rem:Luoassumptions}). Hence, there is no direct implication in either direction between these results and our Theorem \ref{thm: intro connectivity}. Regarding the proof we apply a Mayer-Vietoris
argument for graph cohomology and Poincaré duality. Finally we note that the
orientability condition is necessary in order for Theorem \ref{thm: intro connectivity} to
hold. A counterexample is provided in Example \ref{ex:not2connected}.

The structure of the article is as follows: In the preliminary Section 2 we recall the necessary definitions. Section 3 introduces integration, which is then used in Section 4 to derive Poincaré duality results from orientability. Section 5 is devoted to the orientation covering and in Section 6 the proof of Theorem \ref{thm:orientabilities} is assembled from the previous work. Section 7 contains all results regarding connectivity.

\paragraph*{AI declaration} No generative AI was used in the preparation of this paper.

\paragraph*{Acknowledgments} The authors are very grateful to Emilio Goertsches who, at the age of 9, drew the unlabeled graph underlying Example \ref{ex:not2connected}, thereby pointing us towards the question about its geometric realizability. We are thankful for financial support of the Deutsche Forschungsgemeinschaft
(DFG, German Research Foundation) – 452427095.
The third named author is furthermore supported by the grant PID2024-155646NB-I00 of the Agencia Estatal de Investigación.

\section{Preliminaries}

Throughout this paper we use the coefficient ring $\QQ$ unless stated otherwise. Let us summarize the central notions of GKM theory from the  combinatorial point of view.
As the results of this note deal exclusively with abstract GKM graphs we refrain from
introducing GKM manifolds and their correspondence to GKM graphs. We refer to the survey
 \cite{MR4840842} for the reader interested in learning more about this aspect of the theory.

We consider connected $n$-valent graphs $\Gamma$ without loops, but possibly with multiple edges between vertices. We denote by $V(\Gamma)$ the vertex set of $\Gamma$, and by $E(\Gamma)$ its edge set. The edges of the graph $\Gamma$ are undirected, but for notational reasons we include each edge twice in $E(\Gamma)$, once for each possible orientation, so that for $e\in E(\Gamma)$ its initial and terminal vertex $i(e),t(e)\in V(\Gamma)$ are defined. For $e\in E(\Gamma)$ we denote the edge $e$ with opposite orientation by $\overline{e}$. For a vertex $v\in V(\Gamma)$, we denote by $E_v$ the set of edges $e\in E(\Gamma)$ with $i(e)=v$.
\begin{defn}
	A \emph{connection} $\nabla$ on $\Gamma$ is a collection of bijective maps $\nabla_e:E_{i(e)}\to E_{t(e)}$, $e\in E(\Gamma)$, such that $\nabla_e(e) = \overline{e}$ and $\nabla_{\overline{e}} = (\nabla_e)^{-1}$.
\end{defn}
\begin{defn}\label{defn: abstract GKM graph}
	An \emph{(abstract) (unsigned) GKM graph} is a pair $(\Gamma,\alpha)$, where $\Gamma$ is a connected, $n$-valent graph as above, and $\alpha:E(\Gamma)\to \ZZ^k/\{\pm 1\}$ is an \emph{axial function} on $\Gamma$, i.e., 
	\begin{enumerate}
		\item $\alpha(\overline{e}) = \alpha(e)$ for all $e\in E(\Gamma)$,
		\item {for any $v \in V(\Gamma)$ and $e \neq e' \in E_{v}$, the labels $\alpha(e)$
			and $\alpha(e')$ are linearly independent and}
		\item there exists a \emph{compatible} connection $\nabla$ on $\Gamma$, i.e., for any two edges $e,e'\in E(\Gamma)$ with $i(e) = i(e')$ and for some (hence any) lifts $\widetilde{\alpha}(e')$ of $\alpha(e')$ and $\widetilde{\alpha}(\nabla_e e')$ of $\alpha(\nabla_e e')$ there exists $\varepsilon\in \{\pm 1\}$ such that
		\[
		\widetilde{\alpha}(\nabla_e e') \in \varepsilon \widetilde{\alpha}(e') + \ZZ \alpha(e).
		\]
	\end{enumerate}
\end{defn}
Here and below, a \emph{lift} of an element $\pm a\in \ZZ^k/\{\pm 1\}$ is any of the two elements $a$ and $-a$. 
\begin{rem} \label{rem: signed}
There is also the notion of an (abstract) signed GKM graph, in which the axial function
$\alpha$ takes values in $\ZZ^k$ instead of $\ZZ^k/\{\pm 1\}$. Instead of 1.\ above one
demands $\alpha(\overline{e}) = - \alpha(e)$, and in {3.}\ the sign $\varepsilon$ is assumed to be $+1$ throughout (without having to choose any lifts). This is the original point of view, see \cite{GuilleminZaraEquivariantdeRham}, introduced as a combinatorial model for the one-skeleton of torus actions on almost complex or symplectic manifolds. Unsigned GKM graphs are natural if one is interested in manifolds without such a geometric structure. Signed GKM graphs are orientable in the sense defined below (see Remark \ref{rem: signed orientable}).
\end{rem}
Let us recall the notion of orientability of a GKM graph, introduced in \cite[Section 2.3]{2210.01856v1}.

\begin{defn}\label{D: path orientability}
Consider a GKM graph $(\Gamma, \alpha)$ and choose (only for notational purposes) a lift $\tilde \alpha \colon E(\Gamma)
\to \mathbb{Z}^{k}$ of $\alpha$. Let $v \in V(\Gamma)$ and $e_{1}, \ldots, e_{d}$ the
edges emanating from $v$. Then there are unique $k_{i} \in \mathbb{Z}$ and $\varepsilon_{i}
\in \{\pm 1\}$ $(i=2,\ldots,d)$ such that
\[
	\tilde \alpha\left(\nabla_{e_{1}}(e_{i})\right) = \varepsilon_{i}\tilde \alpha(e_{i}) + k_{i}
	\tilde \alpha(e_{1}).
\]
We set $\eta(e_{1}) = - \prod_{i=2}^{d} \varepsilon_{i}$. The labeled graph $(\Gamma, \alpha)$ is called
\emph{orientable} if for some lift of $\alpha$ (and hence any; cf.\ Remark \ref{rem: orientabilitydef})  we have that for each closed edge path $f_{1},\ldots, f_{m}$
\[
	\prod_{i=1}^{m} \eta(f_{i}) =1.
\]
\end{defn}

\begin{rem}\label{rem: orientabilitydef}
\begin{enumerate}[(i)]
\item It is shown in \cite[Section 2.3]{2210.01856v1} that the product $\prod \eta(f_i)$ in the above definition does not depend on the choice of connection or lift of the axial function. In particular, while we will usually fix a lift in later arguments, there is no preferred choice. 
\item If the graph comes from a GKM manifold, the choice of lift is not equivalent of a choice of orientation. Rather, a choice of lift induces orientations on the representations defined by the labels at each vertex and thus at the tangent representations at each fixed point. In two neighboring vertices, these orientations can be compared through an equivariant tube around a two-sphere belonging to any connecting edge $e$. Whether the orientations match is indicated by $\eta(e)$. Orientability should be viewed as a compatibility condition that ensures that the orientation at one fixed point extends around the whole one-skeleton. A choice of orientation is thus rather represented by a choice of lift together with a choice of vertex. This is indeed what is needed to define integration in Section \ref{sec:ABBV}.
\item For further geometric justification for this terminology, see \cite[Remark 2.20]{2210.01856v1}. In the same remark we asked about the relation between this notion and orientability in the sense of \cite{EscherGoertschesSearle}, i.e., the nonvanishing of top nonequivariant graph cohomology (see below). One purpose of this note is to clarify this relation.
\end{enumerate}
\end{rem}

Although in this note no torus actions on manifolds occur, we picture the target $\ZZ^k$ as the integer lattice in the dual of the Lie algebra of a $k$-dimensional torus $T$ (modulo $\{\pm 1\}$). Let $R$ be the graded polynomial ring $H^*(BT)=\QQ[x_1,\ldots,x_k]$, where each variable $x_i$ has degree $2$. Then $\ZZ^k$ is the integer lattice in the degree $2$ part of $R$. We denote by $R^+\subset R$ the graded maximal ideal in $R$.

\begin{defn}[\cite{MR1823050}]\label{defn:eqgraphcohom}
	For an abstract GKM graph $(\Gamma,\alpha)$ with axial function $\alpha:E(\Gamma)\to \ZZ^k/\{\pm 1\}$ we define its \emph{equivariant graph cohomology (with rational coefficients)} as
	\[
	H^*_T(\Gamma) :=\{(f_v)_{v\in V(\Gamma)}\in R^{V(\Gamma)} \mid \alpha(e) \textrm{ divides } f_{i(e)} - f_{t(e)} \textrm{ for all } e\in E(\Gamma)\}.
	\]
	It naturally is an $R$-algebra, where multiplication in $H^*_T(\Gamma)$ is defined componentwise. The \emph{(nonequivariant) graph cohomology (with rational coefficients)} of $(\Gamma,\alpha)$ is the $\QQ$-algebra
	\[
	H^*(\Gamma):= H^*_T(\Gamma)/R^+\cdot H^*_T(\Gamma).
	\]
\end{defn}
Let us recall the definition of the Thom classes of vertices and edges in equivariant graph cohomology.

\begin{defn}[\cite{GuilleminZara1Skeleta}]
	Let $(\Gamma,\alpha)$ be a $d$-valent GKM graph and $\tilde\alpha$ a lift of $\alpha$. For a vertex $v\in V(\Gamma)$, we define its Thom class ${\mathrm{Th}}_v\in H^{2d}_T(\Gamma)$ as the element $(f(u))_u$ defined by
	\[
	f(u) = \begin{cases}\prod_{e\in E_v} \tilde\alpha(e) & u=v \\ 0 & u\neq v.\end{cases}
	\]
\end{defn}

\begin{defn}[{\cite[Definition 2.15]{2210.01856v1}}]\label{D: Thom class edge}
Let $(\Gamma, \alpha)$ be a $d$-valent GKM graph, $\tilde \alpha$ a lift of $\alpha$ and
$\varepsilon_{i}$, $k_{i}$ as in Definition \ref{D: path orientability}. For an edge $e$
we define the \emph{Thom class} of $e$ by
\[
	(\mathrm{Th}_{e})_{v} = 
	\begin{cases}
	\prod_{f \in E(\Gamma)_{i(e)} \setminus \{e\}\}} \tilde \alpha(f), & v=i(e)\\
	-\eta(e) \cdot \prod_{f \in E(\Gamma)_{t(e)} \setminus \{e\}\}} \tilde \alpha(f), & v=t(e)\\
	0, & \text{otherwise},
	\end{cases}
\]
where $v \in V(\Gamma)$. A straightforward computation shows that $\mathrm{Th}_{e} \in
H^{2d-2}_{T}(\Gamma)$. 
\end{defn}
Note that both types of Thom classes depend on a choice of lift $\tilde{\alpha}$ and are well-defined only up to global sign.

\section{ABBV localization for abstract GKM graphs} \label{sec:ABBV}

This section describes a construction of \cite[Section 2.4]{GuilleminZaraEquivariantdeRham}, with the slight difference that instead of signed GKM graphs we consider  orientable unsigned GKM graphs.

Let $(\Gamma,\alpha)$ be a $d$-valent orientable (unsigned) GKM graph. We fix a vertex
$v_{0}$ and a lift $\tilde\alpha$ of $\alpha$ such that $\eta$ from Definition \ref{D: path
	orientability} is defined. We define an $R$-linear map as follows. For a vertex $w$, we
	choose an edge path $e_1,\ldots,e_n$ from $v_{0}$ to $w$,
and put 
\[
\Psi(w):=\eta(e_1)\cdot\ldots \cdot \eta(e_n)\in \{\pm 1\}
\]
This element $\Psi(w)$ is well-defined by the orientability of $\Gamma$: if we had chosen a different
edge path $f_1,\ldots,f_m$ from $v$ to $w$, then $\eta(e_1)\cdot\ldots\cdot \eta(e_n) =
\eta(f_1)\cdot\ldots\cdot \eta(f_m)$.

Now, for $\omega = (\omega_w)_{w\in V(\Gamma)}$, put
\[
\int \omega:= \sum_{w\in V(\Gamma)} \Psi(w)\frac{\omega_w}{\prod_{e\in E_w} \tilde\alpha(e)}.
\]
which a priori lies in the field of fractions of $R$. However one has
\begin{thm}[{\cite[Theorem 2.2]{GuilleminZaraEquivariantdeRham}}] \label{thm: ABBV}
The above map takes polynomial values and yields a well-defined map \[\int\colon H_T^*(\Gamma)\rightarrow R^{*-2d}.\]
\end{thm}

Here $R$ is understood to be $0$ in negative degrees so in particular the integral vanishes on $H_T^{<2d}(\Gamma)$. The proof is the same as that of \cite[Theorem 2.2]{GuilleminZaraEquivariantdeRham} where the theorem is proved for signed GKM graphs. We only need to modify the proof slightly when dealing with the signs of weights. For the convenience of the reader and due to its essential role in the rest of the paper, we restate the adapted proof here.

\begin{proof}
One can write
\[\int\omega = \frac{g}{\beta_1\cdot\ldots\cdot \beta_N}\]
where the $\beta_i\in H^2(BT;\QQ)$ are labels occurring in the GKM graph.
We may assume that in the equation above the $\beta_i$ are pairwise linearly independent
by leaving out labels that are rational multiples of others, as in the definition of $\int
\omega$ collinear labels never occur in the same denominator. Given the above equation, it suffices to show that each $\beta_i$ divides $g$. For the sake of notation we do the argument for $\beta_1$.

It suffices to prove that $\int \omega$ is a sum of terms of the form
$\frac{f}{\beta_2\cdot\ldots\cdot \beta_N}$ for some polynomial $f\in R$. Fix a vertex
$w$. Clearly, if $\beta_1$ is not (up to rational multiple) an outgoing label at $w$, then
the corresponding term \[\Psi(w)\frac{\omega_w}{\prod_{i(e)=w} \tilde\alpha(e)}\] in the definition of $\int\omega$ is of
this form. If on the other hand $\beta_1=\lambda \tilde{\alpha}(h)$ for some $\lambda\in
\QQ$ and edge $h\in E(\Gamma)$ from $w$ to another $v\in V(\Gamma)$, then we consider the
two summands of $v$ and $w$ together. We denote by $\alpha_1,\ldots,\alpha_d$ and
$\alpha_1',\ldots,\alpha_d'$ the outgoing labels at $w$ and $v$ respectively, where
$\alpha_1=\tilde \alpha(h)=\alpha_1'$. Then the added contributions of $v$ and $w$ to the
integral become
\begin{align*}
\Psi(w)\frac{\omega_w}{\alpha_1\ldots\alpha_d}+\Psi(v)\frac{\omega_v}{\alpha_1'\ldots\alpha_d'}
&= \pm\frac{\omega_w \alpha_2'\ldots\alpha_d'+\eta(h)\omega_v \alpha_2\ldots\alpha_d}{\tilde{\alpha}(h)\alpha_2\ldots\alpha_d\cdot \alpha_2'\ldots\alpha_d'}\\&=
\pm\frac{\omega_w \alpha_2'\ldots\alpha_d'-\omega_v \epsilon_2\alpha_2\ldots\epsilon_d\alpha_d}{\tilde{\alpha}(h)\alpha_2\ldots\alpha_d\cdot \alpha_2'\ldots\alpha_d'}
\end{align*} 
where the $\epsilon_i\in\{ \pm{1}\}$ are (after possibly renumbering the $\alpha_i$) given by $\alpha_i'\equiv \epsilon_i\alpha_i\mod \tilde{\alpha}(h)$. As also $\omega_w\equiv\omega_v\mod \tilde{\alpha}(h)$ we get that the numerator is divisible by $\tilde{\alpha}(h)$. Consequently the above term can also be written as $\frac{f}{\beta_2\cdot\ldots\cdot \beta_N}$, $f\in R$, which finishes the proof.

\end{proof}

As this integration map is $R$-linear, it also induces a well-defined linear map $H^*(\Gamma) \to R/R^+ = \QQ$.

\begin{cor}\label{cor:Thomnonzero}
	For an orientable GKM graph $(\Gamma,\alpha)$, the Thom class $\Th_v$ of any vertex $v$ defines a nonzero element in $H^*(\Gamma)$.
\end{cor}
\begin{proof}
By definition, any Thom class of a vertex integrates to $\pm 1$.
\end{proof}

\begin{rem}\label{rem: signed orientable}
The same corollary holds for signed GKM graphs by the original {\cite[Theorem 2.2]{GuilleminZaraEquivariantdeRham}}.
Conversely, \cite[Proposition 2.23]{2210.01856v1} showed that for a nonorientable GKM graph, the Thom classes of vertices vanish in nonequivariant graph cohomology. This implies in particular that signed GKM graphs are orientable.
\end{rem}

\section{Poincaré duality for GKM graphs}\label{sec:PD}
Let $\Gamma$ be an abstract (unsigned) orientable {$T$}-GKM graph, and $G$ an arbitrary subgraph (not necessarily connected or $n$-valent). Then we define
\[
H^*_T(G):=\{(f_v)_{v\in V(G)}\in R^{V(G)}\mid \alpha(e)\textrm{ divides } f_{i(e)}-f_{t(e)} \textrm{ for all }e\in E(G)\},
\]
where $E(G)\subset E(\Gamma)$ denotes the set of edges belonging to $G$. For $G=\Gamma$ this coincides with the usual equivariant graph cohomology, see Definition \ref{defn:eqgraphcohom} above. This cohomology is contravariant in the sense that for subgraphs $G_1\subset G_2\subset \Gamma$, we have a natural restriction map $H^*_T(G_2)\to H^*_T(G_1)$.

We also define equivariant graph cohomology with compact support, as follows.
\[
H^*_{T,c}(G):= \{(f_v)_{v\in V(G)}\in H^*_T(G)\mid \alpha(e)\textrm{ divides } f(i(e)) \textrm{ for all }e\in E_v\setminus E(G) \}.
\]
We picture $G$, together with small pieces of all edges outgoing of $G$, as a thickening of $G$ inside $\Gamma$. The above cohomology is then the obvious graph-theoretical analogue of ordinary cohomology with compact support.

\begin{rem}
We note that $H_{T,c}^*(G)$ depends on the edges contained in $G$ and does in general not agree with the subring of $H_T^*(\Gamma)$ supported on $V(G)$.
Special cases of the above definition and support conditions in GKM graph cohomology appear in the literature in the context of Thom classes or related to Morse style arguments in a symplectic context. E.g.\ cohomology with support in certain level sets is considered in
\cite[Section 2.9]{GuilleminZaraEquivariantdeRham}.
\end{rem}

\begin{rem}\label{rem:cohomfreedimT=2}
	It was shown in \cite[Proposition 2.25]{2210.01856v1} (and in \cite[Section 3]{Luo}, using arguments from algebraic geometry) that for $T$ a $2$-dimensional torus, the equivariant cohomology $H^*_T(\Gamma)$ is always a free $R$-module. The same argument works for the equivariant cohomologies $H^*_T(G)$ and $H^*_{T,c}(G)$, in case $\dim T = 2$. Concretely, one considers the short exact sequence
	\[
	0\longrightarrow H^*_T(G)\longrightarrow \bigoplus_{v\in V(G)} R\longrightarrow \im \varphi\longrightarrow 0,
	\]
	where $\varphi:\bigoplus_{v\in V(G)} R\to \bigoplus_{e\in E(G)} R/(\alpha(e));\, (f_v)_v\mapsto (f_{i(e)}-f_{t(e)} + R\cdot \alpha(e))_e$. As $\im \varphi$ has depth $\geq 1$, it follows that the depth of $H^*_T(G)$ is $2$, whence it is a free module. The argument for equivariant cohomology with compact support is analogous, with a slightly modified map $\varphi$, for which one takes into account also those edges $e\in E_v\setminus E(G)$.
\end{rem} 
 Just as ordinary cohomology with compact support, this cohomology is covariant, in the
sense that for subgraphs $G_1\subset G_2\subset \Gamma$, we have a natural push-forward
map $H^*_{T,c}(G_1)\to H^*_{T,c}(G_2)$ given by extending a class by zero on vertices in
$V(G_2)\setminus V(G_1)$. In particular, we naturally have $H^*_{T,c}(G)\subset H^*_T(\Gamma)$ for any subgraph $G\subset \Gamma$.

For any $G$ we define the duality map
\[
D_G:H^{*}_T(G)\longrightarrow \Hom_R(H^{*}_{T,c}(G),R);\, \alpha\mapsto (\beta\mapsto \int (\alpha\cdot \beta)),
\]
where we consider the pointwise product $\alpha\cdot \beta\in H^{*}_{T,c}(G)\subset H^{*}_T(\Gamma)$ and $\int:H^{*}_T(\Gamma)\to R$ is the $R$-linear map from Section \ref{sec:ABBV}. Note that  $D_G$ is a map of degree $-2d$ when equipping $\Hom_R(H^{*}_{T,c}(G),R)$ with the grading where the degree $k$ component consists of maps which increase degree by $k$.

 We will prove equivariant Poincaré duality for $G$, as follows:
\begin{thm}\label{thm:PD}
	For orientable $\Gamma$ and any $G\subset \Gamma$, the duality map $D_G$ is an isomorphism.
\end{thm}

\begin{rem} In \cite{AlldayFranzPuppeEqCohomSyzygies}, Allday--Franz--Puppe introduce equivariant homology of a $T$-space as the homology of the dual of the singular Cartan model, and use it to prove an equivariant Poincaré duality statement, see \cite[Corollary 1.3]{AlldayFranzPuppeEqCohomSyzygies}. In their nomenclature, the map $D_G$ corresponds to the cap product with an equivariant orientation. See also \cite{AlldayFranzPuppeEqPALD}. In particular for GKM graphs coming from a GKM manifold, equivariant Poincaré duality in the form of Theorem \ref{thm:PD} is known to hold by using \cite[Corollary 1.3]{AlldayFranzPuppeEqCohomSyzygies} and the fact that non-equivariant Poincaré duality holds for GKM manifolds, which are orientable compact manifolds by definition. In this paper we argue the other way around and deduce Poincaré duality of non-equivariant graph cohomology (Corollary \ref{cor:nonequivPD}) from the equivariant Poincaré duality above. In particular, it is interesting (albeit not terribly useful) to observe that on the geometric side, we obtain a proof of Poincaré duality for GKM manifolds (which always have orientable GKM graphs), which does not invoke the geometric version of the theorem.

The proof of Theorem \ref{thm:PD} is inspired by the standard proof of classical Poincaré duality (see e.g.~\cite[Section 3.3]{MR1867354}).
%It is inspired by \cite{AlldayFranzPuppeEqCohomSyzygies} (Allday-Franz-Puppe) and \cite{Luo}.
\end{rem}

\begin{prop}\label{prop:exactcptandnot} For any subgraph $G\subset \Gamma$  that is the union of two subgraphs $U$ and $W$, we have exact sequences
	\[
	0 \longrightarrow H^*_T(G)\longrightarrow H^*_T(U)\oplus H^*_T(W)\longrightarrow H^*_T(U\cap W)
	\]
	and
	\[
	0 \longrightarrow H^*_{T,c}(U\cap W)\longrightarrow H^*_{T,c}(U)\oplus H^*_{T,c}(W)\longrightarrow H^*_{T,c}(G) \longrightarrow A\longrightarrow 0,
	\]
	where $A$ is a torsion $R$-module.
\end{prop}
\begin{proof}
In the first sequence, the left hand map is given by the difference of the two restrictions, which
is clearly injective. The right hand map is the sum of the the restrictions so the
composition is zero. We observe that an element in the kernel of the right hand map
consists of two elements whose values at the vertices of $U\cap W$ coincide and come from
a single element $x\in R^{V(G)}$. It remains to check that this element satisfies the
congruence relations corresponding to the edges in $E(G)=E(U)\cup E(W)$. But every such
congruence relation is guaranteed by the fact that $x$ restricts to $H_T^*(U)$ and
$H_T^*(W)$.

For the second sequence the first map is the sum of the extensions by zero, whereas the second map
is their difference. Again clearly the first map is injective and the composition is zero.
An element in the kernel of $H^*_{T,c}(U)\oplus H^*_{T,c}(W)\longrightarrow H^*_{T,c}(G)$
comes from a single $x\in R^{V(U\cap W)}$ and it remains to check the congruence
conditions. For every edge in $E(U\cap W)$ the congruence holds from the
fact that $x$ extends to an element of $H_{T,c}^*(U)$. Any edge $e$ of $\Gamma$ adjacent
but not contained in $U\cap W$ is either not contained in $U$ or not contained in $W$ and
again the congruence follows from the fact that $x$ extends to an element of
$H_{T,c}^*(U)$ and of $H_{T,c}^*(W)$. Hence $x\in H_{T,c}^*(U\cap W)$.
	
	It remains to prove that $A$ is torsion. Let $(f_v)_{v\in V(G)}$ be any element in $H^*_{T,c}(G)$. Let $g$ be the product of all labels of edges of $\Gamma$ outgoing of $U$ (with arbitrary sign). Then $(gf_v)_{v\in V_U}$ defines an element of $H^*_{T,c}(U)$ which we can subtract from $(gf_v)_{v\in V(G)}$. The difference is then in $H^*_{T,c}(W)$. We have shown that $(gf_v)_{v\in V(G)}$ is in the image of $H^*_{T,c}(U)\oplus H^*_{T,c}(W)\longrightarrow H^*_{T,c}(G)$.
\end{proof}

\begin{lem}\label{lem:MVSdiagram} There is a commutative diagram of $R$-modules
\[
\begin{tikzcd}[column sep=small, cells={nodes={inner xsep=1pt}}]
	0 \ar[r] & H^{*}_T(G)\ar[r] \arrow{d}{D_G} & H^{*}_T(U)\oplus H^{*}_T(W) \ar[r] \arrow{d}{D_U\oplus D_W} & H^{*}_T(U\cap W)\arrow{d}{D_{U\cap W}} \\
	0 \ar[r] & \Hom(H^{*}_{T,c}(G),R) \ar[r] & \Hom(H^{*}_{T,c}(U),R)\oplus \Hom(H^{*}_{T,c}(W),R) \ar[r] & \Hom(H^{*}_{T,c}(U\cap W),R)
\end{tikzcd}
\]
in which the upper row is exact and the bottom row is a complex, exact at $\Hom_R(H^{*}_{T,c}(G),R)$.
\end{lem}

\begin{proof}
The diagram arises by applying $\Hom_R(-,R)$ to the second sequence in Proposition \ref{prop:exactcptandnot}. It is easily checked to be commutative. The exactness properties also follow from Proposition \ref{prop:exactcptandnot}: this proposition immediately implies that the upper row is exact and that the lower row is a complex; injectivity of \[\Hom_R(H^{*}_{T,c}(G),R) \rightarrow \Hom_R(H^{*}_{T,c}(U),R)\oplus \Hom_R(H^{*}_{T,c}(W),R)\] is assured by the facts that the cokernel $A$ of the original map is torsion and that $R$ is torsion-free.
\end{proof}

\begin{prop}\label{prop:PDbricks}
	The map $D_G$ is an isomorphism when $G$ is 
\begin{enumerate}[(i)]
\item a subgraph with only vertices and no edges.
\item a subgraph consisting only of an edge with its two adjacent vertices.
\end{enumerate}	
\end{prop}
\begin{proof}
	In case $G$ contains only pairwise distinct vertices $v_1,\ldots,v_k$ then an $R$-basis for the free $R$-module $H_T^*(G)$ is given by the elements $1_{v_i}$ which are $1$ at $v_i$ and $0$ everywhere else. An $R$-basis of $H_{T,c}^*(G)$ is given by the Thom classes $\mathrm{Th}_{v_i}$ of the vertices $v_i$. Case $(i)$ now follows from the fact that
	\[\int 1_{v_i}\cdot \mathrm{Th}_{v_j}\]
	is $\pm 1$ if $i=j$ and $0$ if $i\neq j$.
	 Concerning $(ii)$, let $G$ be a subgraph consisting of a single edge $e$ with $i(e)=v$, $t(e)=w$. Then $H^{*}_{T,c}(G)$ is the free $R$-module generated by ${\mathrm{Th}}_v$ and the Thom class $\mathrm{Th}_e$ of the edge $e$: indeed the two are clearly linearly independent. To see that they generate any $f\in H_{T,c}^*(G)$ observe that $f_w$ is divisible by all $\alpha(e')$ for $e\neq e'\in E_w$ and in particular by $(\mathrm{Th}_e)_w$.
	Thus we find some $\beta\in R$ such that $f-\beta \mathrm{Th}_e$ is concentrated at $v$. In particular this difference is a multiple of $\mathrm{Th}_v$. On the other hand $H^{*}_T(G)$ is freely generated by $1$ and the class $g$ with $g_v=0$ and $g_w=\alpha(e)$. Now compute
	\begin{align*}
	\int 1 \cdot \mathrm{Th}_e=0,\quad \int 1\cdot \mathrm{Th}_v =\pm 1,\quad \int g \cdot \mathrm{Th}_e =\pm 1,\quad \int g\cdot \mathrm{Th_v}= 0
	\end{align*}
	which implies $(ii)$.
\end{proof}

\begin{proof}[Proof of Theorem \ref{thm:PD}]
We proceed inductively by starting with a single vertex and successively add edges to it. Write $G=U\cup W$ where by assumption $D_U$ is an isomorphism and $W$ is a single edge graph not contained in $U$. Then $U\cap W$ consists of one or two single vertices. By Proposition \ref{prop:PDbricks} the maps $D_W$ and $D_{U\cap W}$ are isomorphisms. An elementary diagram chase in the diagram of Lemma \ref{lem:MVSdiagram} proves that $D_G$ is an isomorphism.
\end{proof}

The statement also implies ordinary Poincaré duality, in case equivariant graph cohomology is a free module.

\begin{defn}\label{defn: PD}
A graded commutative algebra $A^*$ over a field $k$ is said to satisfy ($n$-dimensional) Poincaré duality if there is an $\epsilon\in \Hom_k(A^n,k)$ such that the map $a\mapsto (b\mapsto \epsilon(ab))$ defines an isomorphism $A^*\cong \Hom_k(A^*,k)$.
\end{defn}
\begin{cor}\label{cor:nonequivPD}
	Let $(\Gamma,\alpha)$ be an orientable $d$-valent GKM graph such that $H^*_T(\Gamma)$ is a free $H^*(BT)$-module. Then $H^*(\Gamma)$ satisfies ($2d$-dimensional) Poincaré duality.
\end{cor}
\begin{proof} For any $R$-module $M$, we write $M/R^+$ for the quotient $M/(R^+\cdot M)$. We identify $\QQ=R/R^+$ and consider the image of $\int$ under the canonical map $p\colon\Hom_R(H_T^*(\Gamma),R)\rightarrow \Hom_\QQ(H^*(\Gamma),\QQ)$. The corresponding duality pairing in the sense of Definition \ref{defn: PD} is the composition
\[H^*(\Gamma)= H_T^*(\Gamma)/R^+\longrightarrow \Hom_R(H_T^*(\Gamma),R)/R^+\longrightarrow \Hom_\QQ (H^{*}(\Gamma),\QQ).\]
where, the first arrow is the map $\overline{D_\Gamma}$ induced by $D_\Gamma$, and the second arrow is the map $\overline{p}$ induced by $p$. It follows from Theorem \ref{thm:PD} that $\overline{D_\Gamma}$ is an isomorphism. Furthermore $\overline{p}$ is an isomorphism because $H_T^*(\Gamma)$ is a free $R$-module.
\end{proof}

Recall that the freeness assumption is automatically satisfied in case $T$ is of dimension $2$ (see Remark \ref{rem:cohomfreedimT=2}). 

\begin{ex} \label{ex: orientable non-free} 
	In case equivariant graph cohomology is not a free module, this statement is not necessarily true. The example of the GKM graph of the flag manifold ${\mathrm{SU}}(3)/T^2$ with $T^3$-labeling given in \cite[Section 5]{GoertschesSolomadin} is a counterexample.  Let us give another example of the same type; the arguments to show orientability and to refute Poincaré duality apply equally to the two examples. Consider the GKM graph
	
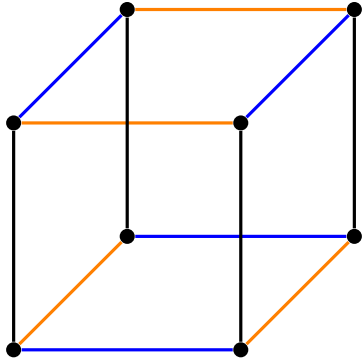
\begin{center}
		\begin{tikzpicture}

			\node (a) at (0,0)[circle,fill,inner sep=2pt] {};
			\node (b) at (3,0)[circle,fill,inner sep=2pt]{};
			\node (c) at (4.5,1.5)[circle,fill,inner sep=2pt]{};	
			\node (d) at (1.5,1.5)[circle,fill,inner sep=2pt]{};
			
			\node (au) at (0,3)[circle,fill,inner sep=2pt] {};
			\node (bu) at (3,3)[circle,fill,inner sep=2pt]{};
			\node (cu) at (4.5,4.5)[circle,fill,inner sep=2pt]{};	
			\node (du) at (1.5,4.5)[circle,fill,inner sep=2pt]{};

			%\node at (3,-1) {$(0,1)$};

			\draw [very thick, blue](a) to (b);
			\draw [very thick, blue](c) to (d);
			\draw [very thick, blue](au) to (du);
			\draw [very thick, blue](cu) to (bu);
			
			\draw [very thick, orange](a) to (d);
			\draw [very thick, orange](c) to (b);
			\draw [very thick, orange](au) to (bu);
			\draw [very thick, orange](cu) to (du);
			
			\draw [very thick](a) to (au);
			\draw [very thick](b) to (bu);
			\draw [very thick](c) to (cu);
			\draw [very thick](d) to (du);
			
		\end{tikzpicture}
\end{center}
where the colors black, blue, and orange correspond to labels via a basis $x,y,z$ of $H^2(BT^3;\ZZ)$. This GKM graph is orientable: we choose for the lift the same sign for the labels of all edges of the same color and the unique connection which respects labels. Then all signs in the congruences will be positive and hence $\eta(e)=-1$ for all $e\in E(\Gamma)$. Since every closed edge path has an even number of edges, orientability follows.

On the other hand $H^*(\Gamma)$ does not satisfy Poincaré duality. This follows e.g.\ by computing $\dim_\QQ H^2(\Gamma)=1$ in combination with the fact that $\dim_\QQ H^*(\Gamma)\geq \rk_R H_T^*(\Gamma)= |V(\Gamma)|=8$. To do the calculation, observe that the 4 congruences within a side face (those involving all three colors) dictate that for a class in $H^2(\Gamma)$,  the $y$ and $z$ coordinates have to be constant within a side face and hence everywhere. Furthermore the $x$ coordinate has to be constant within the top and bottom face respectively. Hence $H^2(\Gamma)$ is generated by the element which is $0$ on all vertices in the top face and $x$ on the vertices in the bottom face.
\end{ex}

\begin{cor}
	For an orientable $d$-valent GKM graph $(\Gamma,\alpha)$ such that $H^*_T(\Gamma)$ is a free $H^*(BT)$-module, $H^{n}(\Gamma)=0$ for $n>2d$.
\end{cor}

\begin{rem}
	Under the more restrictive assumption of a $T^2$-GKM graph immersed in $\RR^2$ in general position (cf.\ Remark \ref{rem:Luoassumptions} below), this statement was shown in \cite[Proposition 4.22]{Luo}. 
\end{rem}

\section{The orientation covering}\label{sec:orientationcovering}

Consider an abstract GKM graph $(\Gamma,\alpha)$. We fix a compatible connection $\nabla$ as well as a lift $\tilde\alpha:E(\Gamma)\to \ZZ^k$ of the axial function.

Define a new graph $\tilde{\Gamma}$ in the following way: its vertex set is
\[
V(\tilde{\Gamma}) = \{(v,\epsilon)\mid v\in V(\Gamma),\ \epsilon=\pm 1\},
\]
and for every edge $e\in E(\Gamma)$ connecting $v$ and $w$ we draw one edge in $\tilde{\Gamma}$ connecting $(v,1)$ with $(w,\eta(e))$ and one edge connecting $(v,-1)$ with $(w,-\eta(e))$.

The graph $\tilde{\Gamma}$ admits a  natural map $\tilde{\Gamma}\to \Gamma$, and we introduce (a lift of) an axial function $\beta$ and 
a connection (again denoted $\nabla$) of $\tilde{\Gamma}$ which intertwines with the (lift of the) axial function and the connection on $\Gamma$.

\begin{thm}\label{thm:orientabilitycoveringoriented}
	The graph $\tilde{\Gamma}$ is connected if and only if $\Gamma$ is not orientable. In this case $\tilde{\Gamma}$ is an orientable GKM graph.
\end{thm} 
\begin{proof}
	Clearly the axial function $\beta$ and the connection $\nabla$ turn $(\tilde{\Gamma},\beta)$ into a GKM graph (up to the connectivity hypothesis such that if $\tilde e,\tilde e'\in E(\tilde \Gamma)$ are the two edges above an edge $e\in E(\Gamma)$, then $\eta(\tilde e)=\eta(e)=\eta(\tilde e')$. The existence of an edge path from $(v,1)$ to $(v,-1)$ in $\tilde{\Gamma}$ is equivalent to a closed loop at $v$ in $\Gamma$ where the $\eta(-)$ multiply to $-1$ along the way. This proves the first part of the statement. It only remains to see that $\tilde{\Gamma}$ is orientable. But by definition of the graph structure, the $\eta(-)$ along a closed loop need to multiply to $+1$.
\end{proof}

\begin{rem}
The orientability covering $\tilde{\Gamma}\rightarrow \Gamma$ does not depend on the choices of connection $\nabla$ and lift $\tilde\alpha$ made in the construction. To see this recall from \cite[Section 2.3]{2210.01856v1} that the product $\prod \eta(e_i)$ along a closed edge loop $c$ does not depend on $\tilde\alpha$ and $\nabla$. We denote this number also by $\eta(c)$. Equipped with this, consider a different $\nabla'$ and $\tilde\alpha'$ and the resulting orientability covering $\tilde{\Gamma}'$. The resulting signs for the edges are denoted $\eta'(-)$.

We fix a base vertex $v_0\in \Gamma$ and define a map $\varphi\colon V(\tilde{\Gamma})\rightarrow V(\tilde{\Gamma}')$ by identifying the fibers over $v_0$ and extending the map via unique path lifting. Concretely, for $(w,\epsilon)\in V(\tilde{\Gamma})$ choose a path $c$ from $w$ to $v_0$ and lift it to $\tilde{\Gamma}$ with starting point $(w,\epsilon)$. Its endpoint is $(v_0,\epsilon\eta(c))$. Now lift the reverse of $c$ to $\tilde{\Gamma}'$ with starting point $(v_0,\epsilon\eta(c))$ and define $\varphi(w,\epsilon)$ to be the endpoint of this path, which is explicitly given by $(w,\epsilon\eta(c)\eta'(c))$. If the map is well-defined then it will define an isomorphism of graphs covering the identity, and therefore preserving the labels.

To see that it is well-defined let $\tilde{c}$ be another path from $w$ to $v_0$. We have to argue that $\eta(c)\eta'(c)=\eta(\tilde{c})\eta'(\tilde{c})$. This is equivalent to $\eta(c)\eta(\tilde{c})=\eta'(c)\eta'(\tilde{c})$ and this holds by the initial discussion, as the concatenation of $c$ and $\tilde{c}$ is a closed edge path.
\end{rem}

{Let us assume that $\Gamma$ is not orientable. Then we} call $\tilde{\Gamma}$, together with the natural projection $\pi:\tilde{\Gamma}\to \Gamma$, the \emph{orientation covering} of $\Gamma$. It admits a unique nontrivial deck transformation $\sigma:\tilde{\Gamma}\to \tilde{\Gamma}$ of order two. It is an automorphism of the GKM graph $\tilde{\Gamma}$ in the sense that it is an automorphism of the underlying graph with the property that it sends every edge to an edge with the same label.

Therefore we can decompose $H^*_T(\tilde{\Gamma})$ into the $\pm 1$ eigenspaces of $\sigma$:
\[H^*_T(\tilde{\Gamma}) = \pi^*H^*_T(\Gamma)\oplus N \cong H^*_T(\Gamma)\oplus N
\] as $H^*(BT)$-modules, where $N=\{\omega\in H^*_T(\tilde{\Gamma})\mid \sigma^*\omega=-\omega\}$. Explicitly, $\pi^*H_T^*(\Gamma)$ consists of those $f$ where $f_{(v,1)}=f_{(v,-1)}$ and $N$ consists of those $f$ where $f_{(v,1)}=-f_{(v,-1)}$ for all $v$.

\begin{lem}\label{lem:coverfree}
If $H_T^*(\Gamma)$ is free then so are $N$ and $H_T^*(\tilde{\Gamma})$. 
\end{lem}

\begin{proof}
As $\tilde{\Gamma}$ is orientable, integration is defined in the sense of Section \ref{sec:ABBV}. We claim that integration vanishes on elements in $\pi^*H_T^*(\Gamma)$. To see this it is enough to observe that $\Psi(w,1)=-\Psi(w,-1)$ with $\Psi$ defined as in Section \ref{sec:ABBV}, depending on a fixed base vertex $(v,1)$. Indeed, this holds since the product of the $\eta(-)$ along any edge path from $(v,1)$ to a vertex $(w,\epsilon)$ is necessarily $\epsilon$.

 We may use equivariant Poincaré duality, and obtain
\[
H^*_T(\Gamma)\oplus N \cong H^*_T(\tilde{\Gamma}) \cong \Hom_R(H^*_T(\tilde{\Gamma}),R)\cong \Hom_R (H^*_T(\Gamma),R)\oplus \Hom_R(N,R).
\]
We note that if $\omega$ and $\omega'$ are either both in $\pi^*H^*_T(\Gamma)$ or both in $N$, then their product lies in $\pi^*H_T^*(\Gamma)$ and hence $\int \omega\cdot \omega'=0$. It follows that the above isomorphism in fact splits as the direct sum of two isomorphisms
\[
H^*_T(\Gamma)\cong \Hom_R(N,R)\qquad \textrm{and} \qquad N\cong \Hom_R(H^*_T(\Gamma),R).
\]
So in particular, if $H^*_T(\Gamma)$ is a free module, also $N$ and hence also $H^*_T(\tilde{\Gamma})$ are free modules.
\end{proof}

\section{Orientability conditions}

We can now complete the

\begin{proof}[Proof of Theorem \ref{thm:orientabilities}]
	We proved in Theorem \ref{thm:PD} that for orientable $\Gamma$, equivariant Poincaré duality holds true, and as $H^*_T(\Gamma)$ is a free module, Corollary \ref{cor:nonequivPD} implies that $H^*(\Gamma)$ satisfies nonequivariant Poincaré duality as well. In particular, $\dim H^{2d}(\Gamma)=\dim H^0(\Gamma)=1$, and by Corollary \ref{cor:Thomnonzero} this vector space is generated by the Thom class of any vertex of $\Gamma$. This implies that Condition 1 implies all other three conditions.
	
	It remains to show that for nonorientable $\Gamma$, $H^{2d}(\Gamma)=0$, as then none of the conditions 2-4 can hold true. To this end, we consider the orientation covering $\pi:\tilde{\Gamma}\to \Gamma$ (after fixing a lift $\tilde\alpha$ of $\alpha$). We have that $\tilde{\Gamma}$ is orientable by Theorem \ref{thm:orientabilitycoveringoriented}, and has free equivariant cohomology, as shown in Lemma \ref{lem:coverfree}. So by the above, it follows that $H^{2d}(\tilde{\Gamma})\cong \QQ$, generated by the Thom class of any vertex. With respect to the decomposition $H_T^*(\tilde{\Gamma}) = \pi^* H_T^*(\Gamma)\oplus N$ from Section \ref{sec:orientationcovering} we get
	
\[\mathrm{Th}_{(v,1)}=\frac{\pi^*\mathrm{Th}_v}{2}+\frac{\mathrm{Th}_{(v,1)}-\mathrm{Th}_{(v,-1)}}{2}.\]
 Now the first summand projects to $0$ in $H^*(\tilde\Gamma)$ as it comes from $H_T^*(\Gamma)$ and the Thom class of a vertex in a nonorientable GKM graph vanishes in nonequivariant graph cohomology by \cite[Proposition 2.23]{2210.01856v1}. In particular this implies that $N/(R^+\cdot N)$ has a generator in degree $2d$ while $H^*(\Gamma)$ has none.
\end{proof}

\begin{rem}\label{rem:Luo1}
	The fact that top nonequivariant graph cohomology $H^{2d}(\Gamma)$ is either zero- or one-dimensional was, under the more restrictive assumption of a $T^2$-GKM graph immersed in $\RR^2$ in general position (cf.\ Remark \ref{rem:Luoassumptions} below), shown in \cite[Proposition 4.22]{Luo}. 
\end{rem}

\section{Connectivity of GKM graphs}

Recall the following definitions:
\begin{defn}
	A connected graph $\Gamma$ is \emph{$k$-edge connected} if it has the property that after the removal of any $l$ edges, $l< k$, it remains connected. It is \emph{$k$-vertex connected} if it has the property that after the removal of any $l$ vertices and its adjacent edges, $l<k$, it remains connected.
\end{defn}

\begin{lem}\label{lem:2vertex2edge}
	A $2$-vertex-connected graph different from the graph with two vertices and a single edge connecting them is also $2$-edge-connected.
\end{lem}
\begin{proof}The claim is clear for a graph with only two vertices (and arbitrarily many edges), so let $\Gamma$ be a $2$-vertex-connected graph with more than two vertices. Let $e$ be an edge of $\Gamma$, connecting vertices $v$ and $w$. By assumption, there is another edge $e'$ at $v$, with terminal vertex $w'\neq w$. As $\Gamma$ remains connected when removing $v$ and all its adjacent edges, there is a path in $\Gamma$ from $w$ to $w'$ avoiding all edges adjacent to $v$. Concatenating this path with $e'$ we have found a path connecting $v$ and $w$ avoiding $e$, hence $\Gamma$ is $2$-edge-connected.
\end{proof}

\begin{thm}\label{thm:connectivity}
	Any orientable $d$-valent GKM graph $(\Gamma,\alpha)$ is $2$-vertex-connected. For $d\geq 2$ it is also $2$-edge-connected.
\end{thm}

\begin{proof}[Proof of Theorem \ref{thm:connectivity}]
	The claim is obvious for $d=2$, so let us assume that $d\geq 3$. By Lemma \ref{lem:2vertex2edge} we only need to consider $2$-vertex-connectivity. As $2$-vertex-connectivity is unchanged when restricting to a smaller torus, we may assume that we consider $T^2$-labelings.
	
	Assume that $(\Gamma,\alpha)$ is an orientable $d$-valent GKM graph that is not $2$-vertex-connected, i.e., it has a vertex $v_0$ such that upon removal of $v_0$ and all its adjacent edges, $\Gamma$ decomposes into several connected components. In this situation we find two proper subgraphs $U$ and $W$ of $\Gamma$ with the property that $U\cap W$ consists only of the vertex $v_0$ {while $U,W$ are themselves connected}. The corresponding Mayer-Vietoris sequence (see Proposition \ref{prop:exactcptandnot})
	\[
	0 \longrightarrow H^*_T(\Gamma)\longrightarrow H^*_T(U)\oplus H^*_T(W)\longrightarrow H^*_T(U\cap W)\longrightarrow 0
	\]
	is short exact; indeed, $1\in H^*_T(U\cap W) \cong H^*(BT)$ is the image of $(1,0)\in H^*_T(U)\oplus H^*_T(W)$.
	
	As all modules involved are free (graded) $H^*(BT)$-modules (this uses that we are considering $T^2$-labelings, see Remark \ref{rem:cohomfreedimT=2}) the sequence splits. But this leads to a contradiction as, by Theorem \ref{thm:orientabilities}, the equivariant cohomology $H^*_T(\Gamma)$ of the orientable GKM graph $\Gamma$ has a generator in degree $2d$, while neither $H^*_T(U)$ nor $H^*_T(W)$ have a generator in this degree. The last claim is proved as follows: Observe first that $H_{T,c}^*(U)$ is trivial in degree $0$ because $U$ is connected but does not contain all edges emanating from $v_0$. Consequently, its dual module $\Hom_R(H^*_{T,c}(U),R)$ does not have a generator in degree $0$ (for a free graded $R$-module $M$ with generators in degrees $n_i$, its dual module $\Hom_R(M,R)$ has generators in degrees $-n_i$). By the Poincaré duality Theorem \ref{thm:PD}, $H^*_T(U) \cong \Hom_R(H^*_{T,c}(U),R)$ via the duality isomorphism $D_U$ of degree $-2d$, and it follows that $H^*_T(U)$ has no generators in degree $2d$.
\end{proof}

\begin{ex}\label{ex:not2connected}
	The following is an example of a non-orientable GKM graph that is not $2$-edge connected, which shows that the orientability assumption in Theorem \ref{thm:connectivity} is necessary.
	\begin{figure}[H]
		\begin{center}
			\begin{tikzpicture}
				\draw[very thick] (0,0) to[in = 105, out = 255] (0,-3);
				\draw[very thick] (0,0) to[very thick, in = 75, out = 285] (0,-3);
				\draw[very thick] (0,0) -- ++(1.5,-1.5) -- ++(-1.5,-1.5);
				\draw[very thick] (1.5,-1.5) -- ++ (2,0) -- ++ (1.5,1.5) -- ++(-1.5,-1.5) -- ++ (1.5,-1.5);
				\draw[very thick] (5,0) to[in = 105, out = 255] (5,-3);
				\draw[very thick] (5,0) to[very thick, in = 75, out = 285] (5,-3);
				\node at (.75,-1.5) {$(0,1)$};
				\node at (-.75,-1.5) {$(2,0)$};
				\node at (1.4,-.5) {$(1,-1)$};
				\node at (1.25,-2.5) {$(1,1)$};
				\node at (2.5,-1.1) {{$(1,0)$}};
				\node at (3.65,-.5) {$(1,-1)$};
				\node at (3.75,-2.5) {$(1,1)$};
				\node at (4.25,-1.5) {$(0,1)$};
				\node at (5.75,-1.5) {$(2,0)$};

				\node at (0,0)[circle, fill, inner sep = 2pt] {};
				\node at (0,-3)[circle, fill, inner sep = 2pt] {};
				\node at (1.5,-1.5)[circle, fill, inner sep = 2pt] {};
				\node at (3.5,-1.5)[circle, fill, inner sep = 2pt] {};
				\node at (5,0)[circle, fill, inner sep = 2pt] {};
				\node at (5,-3)[circle, fill, inner sep = 2pt] {};
			\end{tikzpicture}
		\end{center}
		\caption{A non-orientable GKM graph that is not $2$-edge connected}
		\label{fig:not2connected}
	\end{figure}
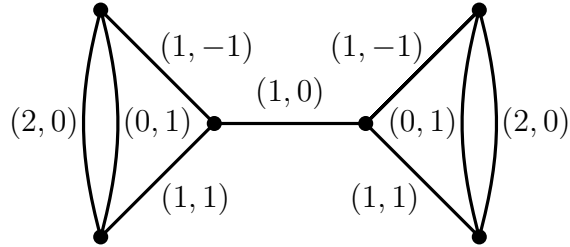

\end{ex}
\begin{ex}\label{ex:not3connected}
	Consider $\Gamma_0$, the ($n$-valent) GKM graph of a $T^2$-GKM action on $\CC P^{n}$. As an unlabeled graph, it is the complete graph on $n+1$ vertices. Let $\Gamma$ be the $(n+1)$-valent GKM graph which is constructed from four copies of $\Gamma_0$, as Figure \ref{fig:not3connected} demonstrates for the case $n=3$. The edges not contained in the four copies of $\Gamma_0$ shall all have the same label, which is arbitrary except that it shall be pairwise linearly independent from any of the labels of $\Gamma_0$. Moreover, they shall connect vertices whose adjacent edges carry the same  labels.
	
	\begin{figure}[H]
		\begin{center}
			\begin{tikzpicture}
				
				\node(a) at (0,0)[circle,fill,inner sep=2pt] {};
				\node(b) at (-2,.2)[circle,fill,inner sep=2pt] {};
				\node(c) at (-1.2,1.7)[circle,fill,inner sep=2pt] {};
				\node(d) at (-1.5,-.8)[circle,fill,inner sep=2pt] {};
				
				\node(a') at (0,-4)[circle,fill,inner sep=2pt] {};
				\node(b') at (-2,.2-4)[circle,fill,inner sep=2pt] {};
				\node(c') at (-1.2,1.7-4)[circle,fill,inner sep=2pt] {};
				\node(d') at (-1.5,-.8-4)[circle,fill,inner sep=2pt] {};
				
				\node(a'') at (3,0)[circle,fill,inner sep=2pt] {};
				\node(b'') at (3+2,.2)[circle,fill,inner sep=2pt] {};
				\node(c'') at (3+1.2,1.7)[circle,fill,inner sep=2pt] {};
				\node(d'') at (3+1.5,-.8)[circle,fill,inner sep=2pt] {};
				
				\node(a''') at (3,-4)[circle,fill,inner sep=2pt] {};
				\node(b''') at (3+2,.2-4)[circle,fill,inner sep=2pt] {};
				\node(c''') at (3+1.2,1.7-4)[circle,fill,inner sep=2pt] {};
				\node(d''') at (3+1.5,-.8-4)[circle,fill,inner sep=2pt] {};

				\draw [very thick](a) to (b) to (c) to (d) to (a) to (c);
				\draw[very thick] (b) to (d);
				
				\draw [very thick](a') to (b') to (c') to (d') to (a') to (c');
				\draw[very thick] (b') to (d');
				
				\draw [very thick](a'') to (b'') to (c'') to (d'') to (a'') to (c'');
				\draw[very thick] (b'') to (d'');
				
				\draw [very thick](a''') to (b''') to (c''') to (d''') to (a''') to (c''');
				\draw[very thick] (b''') to (d''');
				
				\draw[very thick](a) to (a'');
				\draw[very thick](a') to (a''');
				\draw[very thick](b) to (b');
				\draw[very thick](c) to (c');
				\draw[very thick](d) to (d');
				\draw[very thick](b'') to (b''');
				\draw[very thick](c'') to (c''');
				\draw[very thick](d'') to (d''');

			\end{tikzpicture}
		\end{center}
		\caption{An orientable GKM graph that is not $3$-edge connected}
		\label{fig:not3connected}
	\end{figure}
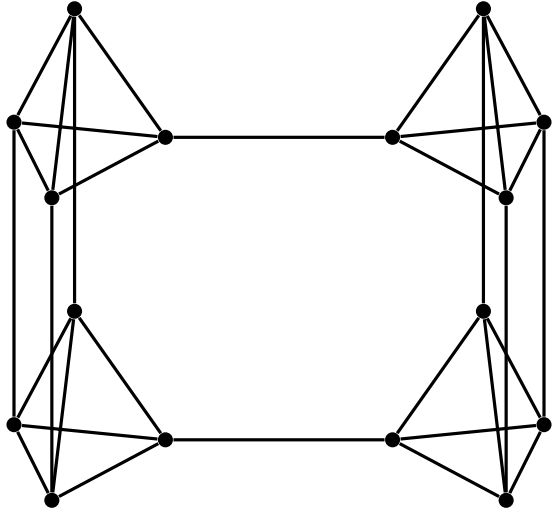
	The GKM graph $\Gamma$ is orientable and $2$-edge connected, but not $3$-edge connected.	
	Note that for $n=2$ the orientability covering (see Section \ref{sec:orientationcovering}) of Example \ref{ex:not2connected} is of this form.
\end{ex}

\begin{rem}\label{rem:Luoassumptions}
	Example \ref{ex:not3connected} shows that Luo's connectivity results \cite[Corollary 5.4 and 5.6]{Luo}, proven for certain Hamiltonian GKM graphs, do not generalize to arbitrary (orientable) GKM graphs. Motivated by Hamiltonian geometry, Luo considers in these theorems graphs that are immersed into $\RR^2$, with edge labels given by the respective slopes. In addition, he only considers graphs that are \emph{in general position}, which for him means that the image of no three vertices in $\RR^2$ are colinear -- this condition is not to be confused with the standard notion of being in general position (see e.g.~\cite{AyzenbergMasuda}) This global condition is restrictive, and does not seem to have an evident geometric meaning. For instance, Figure \ref{fig:notingenpos} is an immersion of Example \ref{ex:not3connected} above for $n=2$ in a way compatible with the slopes, but which is not in general position.
	\begin{figure}[H]
		\begin{center}
			\begin{tikzpicture}
			
			\node(a) at (0,0)[circle,fill,inner sep=2pt] {};
			\node(b) at (1.5,-1.5)[circle,fill,inner sep=2pt] {};
			\node(c) at (1.5,1.5)[circle,fill,inner sep=2pt] {};
			\node(d) at (2.25,.75)[circle,fill,inner sep=2pt] {};
			\node(e) at (2.25,-.75)[circle,fill,inner sep=2pt] {};
			\node(f) at (3,0)[circle,fill,inner sep=2pt] {};
			\node(g) at (5,0)[circle,fill,inner sep=2pt] {};
			\node(h) at (8,0)[circle,fill,inner sep=2pt] {};
			\node(i) at (6.5,-1.5)[circle,fill,inner sep=2pt] {};
			\node(j) at (6.5,1.5)[circle,fill,inner sep=2pt] {};
			\node(k) at (5.75,-.75)[circle,fill,inner sep=2pt] {};
			\node(l) at (5.75,.75)[circle,fill,inner sep=2pt] {};

			\draw[very thick](a) to (b) to (c) to (a) to (f) to (d) to (e) to (k) to (l) to (g) to (h) to (i) to (j) to (c);
			\draw[very thick] (h) to (j);
			\draw[very thick] (b) to (i);
			\draw[very thick] (e) to (f);
			\draw[very thick] (h) to (j);
			\draw[very thick] (d) to (l);			
			\draw[very thick] (g) to (k);

			\end{tikzpicture}
		\end{center}
		\caption{An immersion, not in general position}
		\label{fig:notingenpos}
	\end{figure}
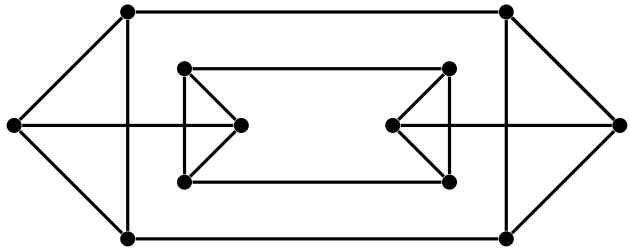
\end{rem}

\bibliographystyle{acm}
%\printbibliography
%\bibliography{/home/pako/.config/TeXFiles/master.bib}
\bibliography{orientability}

\end{document}